\documentclass[ejp]{imsart}

\RequirePackage[numbers]{natbib}
\RequirePackage{amsthm,amsmath}
\usepackage{pictexwd}
\usepackage{amssymb,graphicx,url}
\usepackage[colorlinks,citecolor=blue,urlcolor=blue]{hyperref}
\usepackage{bm}

\startlocaldefs    

\def\argmax{\mbox{argmax}}

\def\a{\alpha}
\def\b{\beta}
\def\R{\mathbb R}

\def\d{\delta}

\def\E{{\mathbb E}}
\def\F{{\mathbb F}}

\def\P{{\mathbb P}}
\def\g{\gamma}

\def\labda1{\lambda_1}
\def\labda2{\lambda_2}

\def\m{\mu}
\def\n{\nu}
\def\e{\varepsilon}
\def\f{\phi}

\def\comment#1{\relax}

\def\=in{\mathop{\rm =}}

\newtheorem{theorem}{Theorem}[section]

\newtheorem{lemma}{Lemma}[section] 

\newtheorem{example}{Example}[section]
\newtheorem{remark}{Remark}[section]

\newtheorem{condition}{Condition}[section]

\numberwithin{equation}{section}
\theoremstyle{plain}
\setattribute{journal}{name}{}
\endlocaldefs

\usepackage{amsmath,here,amsfonts,latexsym,graphicx,here}
\usepackage{amsthm,curves}
\usepackage{graphicx}
\usepackage{caption}
\usepackage{subcaption}
\usepackage{float}

\advance \topmargin by -\headheight \advance \topmargin by
-\headsep

\def\a{\gamma}
\def\m{\mu}
\def\n{\nu}

\def\GG{\Gamma}
\def\th{\theta}

\def\P{{\mathbb P}}
\def\F{{\mathbb F}}

\def\bmA{\bm A}

\def\bmz{\bm z}

\def\bms{\bm s}

\def\a{\alpha}

\begin{document}
\begin{frontmatter}
\title{Grenander functionals and Cauchy's formula}
\runtitle{Grenander functionals}

\begin{aug}
\runauthor{Piet Groeneboom}
\author{\fnms{Piet} \snm{Groeneboom}\corref{}\ead[label=e1]{P.Groeneboom@tudelft.nl}
\ead[label=u1,url]{http://dutiosc.twi.tudelft.nl/\textasciitilde pietg/}}
\address{Delft University of Technology, Building 28, Van Mourik Broekmanweg 6, 2628 XE Delft, The Netherlands.\\ \printead{e1}}
\end{aug}

\thankstext{t2}{This manuscript is dedicated to the memory of Ronald Pyke}

\begin{abstract}
Let $\hat f_n$ be the nonparametric maximum likelihood estimator of a decreasing density. Grenander characterized this in \cite{Grenander:56} as the left-continuous slope of the least concave majorant of the empirical distribution function. For a sample from the uniform distribution, the asymptotic distribution of the $L_2$-distance of the Grenander estimator to the uniform density was derived in \cite{piet_ron:83} by using a representation of the Grenander estimator in terms of conditioned Poisson and gamma random variables. This representation was also used in \cite{GroLo:93} to prove a central limit result of Sparre Andersen \cite{sparre:54} on the number of jumps of the Grenander estimator. Here we extend this to the proof of the main result in \cite{piet_ron:83} and also prove a similar asymptotic normality results for the entropy functional. In \cite{piet_ron:83} the limit distribution of the sums of gamma and Poisson variables on which the conditioning was done did not have the right form, which is corrected here. Cauchy's formula  and saddle point methods are the main tools in our development.
\end{abstract}

\begin{keyword}[class=AMS]
\kwd[Primary ]{62E20, }
\kwd[Secondary ]{62G05}
\end{keyword}

\begin{keyword}
\kwd{Grenander estimator, integral statistics, saddle points, Cauchy's formula}
\end{keyword}

\end{frontmatter}

\section{Introduction}
\label{sec:intro}
The Grenander estimator is the (nonparametric) maximum likelihood estimator (MLE) of a monotone decreasing density. It was introduced in \cite{{Grenander:56}}, where it was proved that it is the left-continuous slope of the least concave majorant of the empirical distribution function. Some properties and limit results  are discussed in \cite{piet_geurt:14} and also in \cite{piet_geurt:18} in the special issue on nonparametric inference under shape constraints of the journal Statistical Science. The Grenander estimator is shown in Figure \ref{fig:picture_Grenander} for a sample of size $n=100$ from the uniform distribution on $[0,1]$. It can be improved by using boundary penalties (in fact, the estimator is inconsistent at the boundary points $0$ and $1$), but this is not the concern of the present paper.

\begin{figure}[!ht]
	\centering
	\begin{subfigure}{0.45\linewidth}
		\includegraphics[width=0.95\textwidth]{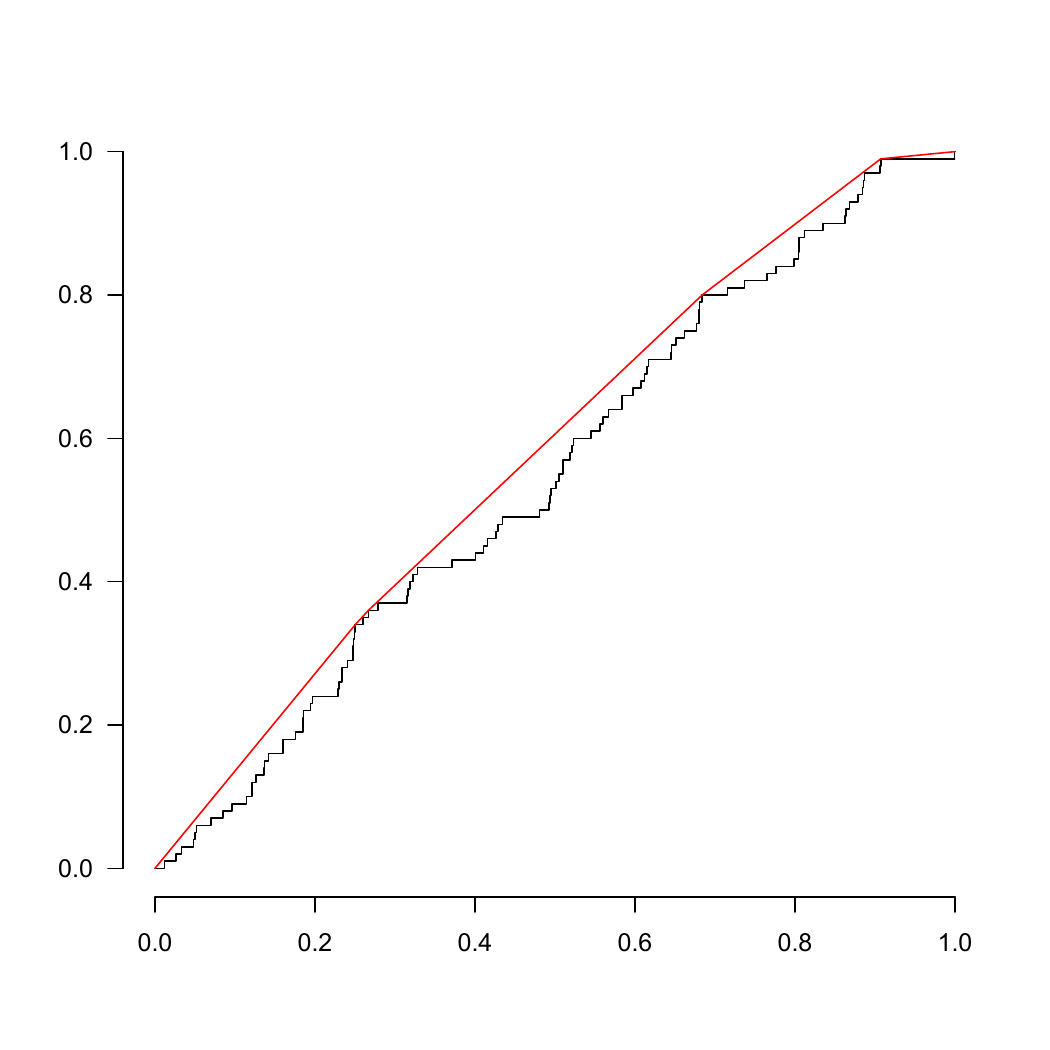}
		\caption{empirical distribution function (black) and its least concave majorant (red)}
	\end{subfigure}
	\begin{subfigure}{0.45\linewidth}
	\includegraphics[width=0.95\textwidth]{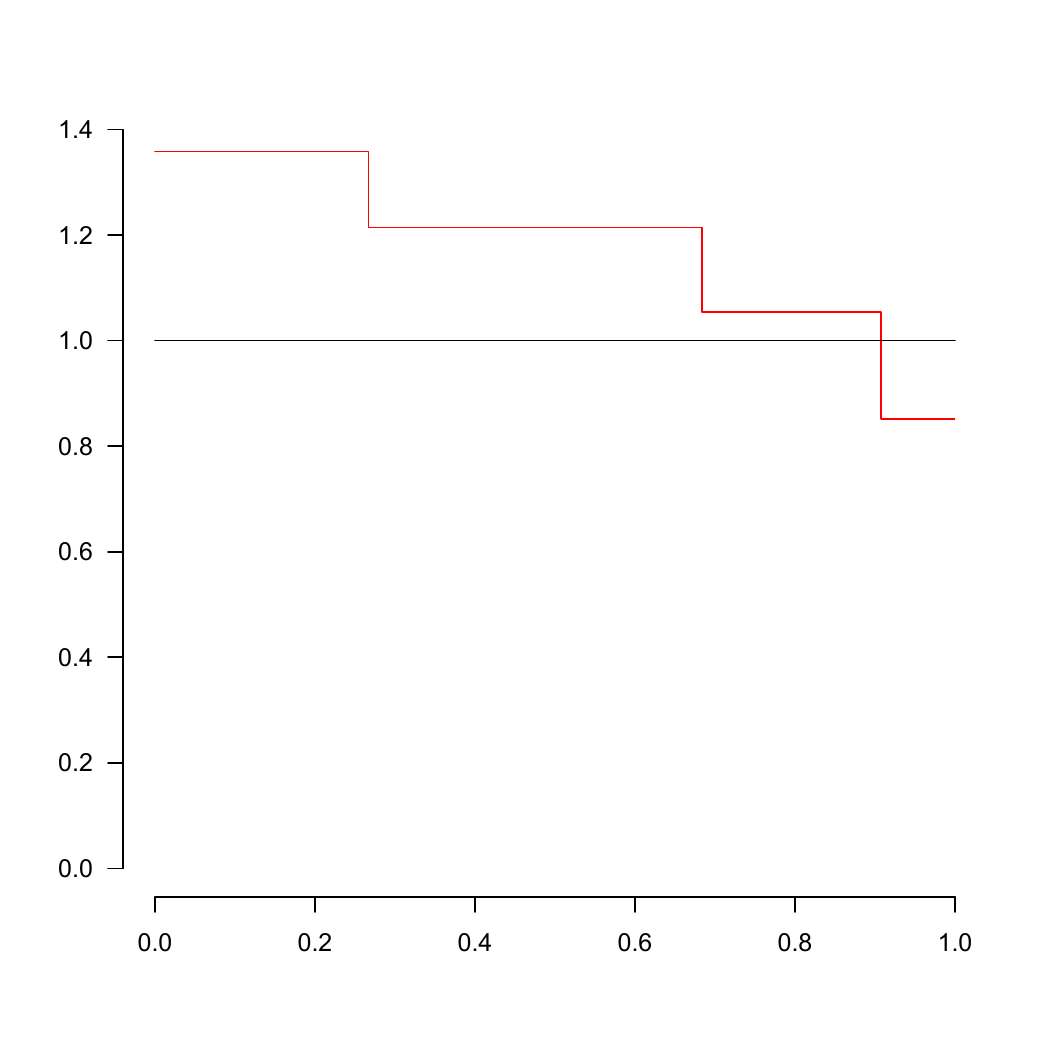}
	\caption{Grenander estimator}
	\end{subfigure}
	\caption{Left: the empirical distribution function and its least concave majorant and right: the Grenander estimator, on the basis of a sample of size $n=100$ from the uniform distribution on $[0,1]$}
	\label{fig:picture_Grenander}
\end{figure}

The Grenander estimator is a piecewise constant function with downward jumps at locations which correspond to the changes of slope (``kinks") of the least concave majorant of the empirical distribution function. Although the Grenander estimator is defined as the left-continuous slope of the empirical distribution function, we can make the Grenander estimator right-continuous by taking the limits on the right at its points of jump. This does not change the probability mass of the induced (absolutely continuous) probability distribution, which is absolutely continuous w.r.t.\ Lebesgue measure.

The number of jumps of the Grenander estimator is of order $\log n$ if the sample is from a uniform distribution (see Section \ref{sec:Sparre_Andersen}), if the sample comes from a strictly decreasing smooth density like the exponential density, then the number of jumps is of order $cn^{1/3}$, for some constant $c>0$. The limit behavior of the Grenander estimator for these situations is rather different. For a sample from the uniform distribution, we have for $t\in(0,1)$:
\begin{align}
\label{limit_Grenander_uniform}
\sqrt{n}\left\{\hat f_n(t)-1\right\}\stackrel{{\cal D}}\longrightarrow S_t,
\end{align}
where $S_t$ is the slope of the least concave majorant of the standard Brownian Bridge on $[0,1]$, see Remark 2.2, p.\ 543 of \cite{piet:85}. The density of $S_t$ is a function of the standard normal distribution function $\Phi$ and the standard normal density $\f$, see (3.11) in \cite{piet:83}.

In contrast with the result (\ref{limit_Grenander_uniform}), we get for a sample from a decreasing density $f$ on $[0,\infty)$ at a point $t\in(0,\infty)$, where $f$ is differentiable and $f'(t)<0$, the following result, due to Prakasa Rao in \cite{prakasa:69}:
\begin{align}
\label{limit_Grenander_decreasing}
n^{1/3}\left|4f(t)f'(y)\right|^{-1/3}\left\{\hat f_n(t)-f(t)\right\}\stackrel{{\cal D}}\longrightarrow Z,
\end{align}
where $Z=\argmax_t\{W(t)-t^2\}$, that is: $Z$ is the (almost surely unique) \text{\rm location of the maximum} of two-sided Brownian motion minus the parabola $y(t)=t^2$.
For further details, see, e.g., \cite{piet_geurt:14} and \cite{piet_geurt:18}.

Recently, integrated functionals of a monotone density were studied in \cite{bodhi:19}. Using the same notation as in \cite{bodhi:19} the following functionals were studied:
\begin{align*}
\m(h,f)=\int_0^1 h(f(x))\,dx,
\end{align*}
where $f$ is a nonincreasing function on $\R_+$ and $h$ satisfies some regularity conditions.
In the case that the underlying distribution is uniform, the following central limit result is proved:

\begin{theorem}[Theorem 3.2 in \cite{bodhi:19}]
\label{th:bodhi}
Let $h\in C^4([0,\infty))$ and let $h''(1)\ne0$. Then:
\begin{align}
\label{central_limit_uniform}
\frac{n\left\{\m(h,\hat f_n)-\m(h,f)\right\}-\tfrac12h''(1)\log n}{\sqrt{\tfrac34h''(1)^2\log n}}\stackrel{{\cal D}}\longrightarrow N(0,1),
\end{align}
where $N(0,1)$ denotes the standard normal distribution.
\end{theorem}

We prove an analogous result for analytic functions $h$, defined on the positive open complex half plane. So our functions $h$ have a lot more smoothness, but are on the other hand defined on the {\it open} complex half plane, which makes the result applicable to functions that are not covered by the conditions in \cite{bodhi:19}. We assume that $h$ satisfies the following condition.

\begin{condition}
\label{condition_h}
{\rm
The function $h$ is analytic on the complex half plane $\{z\in\mathbb C:\text{Re}(z)>0\}$, and satisfies the following conditions.
\begin{enumerate}
\item[(i)] $h''(1)\ne0$.
\item[(ii)] 
For $t\in\R$:
\begin{align}
\label{condition1} 
\left|h(1-it)\right|=O(t^2),\qquad |t|\to\infty,
\end{align}
and
\item[(iii)]
For $t\in\R$:
\begin{align}
\label{condition2}
\left|h''(1-it)\right|=O(1),\qquad |t|\to\infty.
\end{align}
\end{enumerate}
}
\end{condition}

We now have the following result.

\begin{theorem}
\label{main_theorem}
Let $h$ satisfy Condition \ref{condition_h} and let $\hat f_n$ be the Grenander estimator. Then
\begin{align}
\label{central_limit_uniform2}
\frac{n\left\{\m(h,\hat f_n)-\m(h,f)\right\}-\tfrac12h''(1)\log n}{\sqrt{\tfrac34h''(1)^2\log n}}\stackrel{{\cal D}}\longrightarrow N(0,1),
\end{align}
where $N(0,1)$ denotes the standard normal distribution.
\end{theorem}

The result is a corollary to Theorem \ref{th:main_theorem} in Section \ref{sec:bodhi}, see the remark at the end of Section \ref{sec:bodhi}.

\vspace{0.3cm}
Examples of  the application of Theorem \ref{main_theorem} are:

\begin{example}
\label{results_piet_ron}
{\rm
Let $h(z)=(z-1)^2$. Then $h'(z)=2(z-1)$ and $h''(z)=2$. The function $h$ is obviously analytic on the positive complex half plane.
Condition \ref{condition_h} is fulfilled, so we get
\begin{align}
\label{main_result_piet_ron}
\frac{1}{\sqrt{3\log n}}\left\{n\int_0^1\bigl\{\hat f_n(t)-1\bigr\}^2\,dt-\log n\right\}\stackrel{{\cal D}}\longrightarrow N(0,1),\qquad n\to\infty.
\end{align}
This is the main result of \cite{piet_ron:83}. Since in \cite{bodhi:19} Theorem 3.2 is deduced from this result, we also get Theorem 3.2 in \cite{bodhi:19} back from Theorem \ref{main_theorem}.
}
\end{example}

\begin{example}
\label{results_entropy_functional}
{\rm
Let $h(z)=z\log z$. Then $h'(z)=1+\log(z)$ and $h''(z)=1/z$. The function $h$ is again analytic on the positive complex half plane.
Condition \ref{condition_h} is fulfilled, and we get:
\begin{align}
\label{entropy_func}
\frac{1}{\sqrt{\tfrac34\log n}}\left\{n\int_0^1\hat f_n(t)\log\hat f_n(t)\,dt-\tfrac12\log n\right\}\stackrel{{\cal D}}\longrightarrow N(0,1),\qquad n\to\infty.
\end{align}
For this example the conditions of  Theorem 3.2 in \cite{bodhi:19} are not satisfied. The result follows from Theorem \ref{main_theorem} and can be applied to the theory on a likelihood ratio test for monotonicity in \cite{chan:18}.
}
\end{example}

To derive limit results for the uniform distribution, a special representation in terms of gamma and Poisson random variables was given in \cite{piet_ron:83}, with the purpose of proving a limit result for a two-sample rank statistic introduced in the dissertation of \cite{behnen:74} and also for a test statistic in the combination of tests in \cite{scholz:83}. We describe this representation now.

Let $X_1,\dots,X_n$ be a sample from the uniform distribution, and let $0=\xi_{n0}<\xi_{n1}<\dots<\xi_{nm}<\xi_{n,m+1}=1$ be the locations of the jumps of the Grenander estimator $\hat f_n$ for this sample, augmented with the points $0$ and $1$. Note that $[\xi_{n,0},\xi_{n,1}]$, $(\xi_{n,1},\xi_{n2}]$, $(\xi_{n,2},\xi_{n3}]$, $\dots$, $(\xi_{n,m},1]$ are the successive intervals of constancy of the Grenander estimator if we take the estimator to be left-continuous.

 Furthermore, let $D_{ni}$, $J_{ni}$ and $Q_{nj}$ be defined by:
\begin{align*}
\begin{array}{ll}
D_{ni}=\xi_{ni}-\xi_{n,i-1}\qquad &,i=1,\dots,m+1,\\
J_{ni}=n\left\{\F_n(\xi_{ni})-\F_n(\xi_{n,i-1})\right\}\qquad &,i=1,\dots,m+1,\\
Q_{nj}=\#\left\{i:J_{ni}=j\right\}, &
\end{array}
\end{align*}
where $\F_n$ is the empirical distribution function of the sample $X_1,\dots,X_n$, and where $m$ is the number of jumps of the Grenander estimator.

Next, let $\{N_j:j\ge1\}$ be independent Poisson random variables with $\E N_j=1/j$, and let, for each $i$, $\{S_{ji},\,i,j\ge1\}$ be a collection of independent gamma random variables, independent of the $N_j$, where $S_{ji}$ is Gamma$(j,1)$ (the sum of $j$ independent standard exponentials). We define:
\begin{align}
\label{def_S_n_T_n}
S_n=\sum_{j=1}^n\sum_{i=1}^{N_j}S_{ji},\qquad T_n=\sum_{j=1}^njN_j.
\end{align}
and
\begin{align*}
\bm S^{(n)}=\left(S_{11},\dots,S_{1,N_1},\dots,S_{n1},\dots,S_{n,N_n}\right),\qquad\qquad {\bm N}^{(n)}=\left(N_1,\dots,N_n\right).
\end{align*}
Note that there are $N_1$ induced spacings $\xi_{ni}$ (intervals of constancy of $\hat f_n$) of length $1$, $N_2$ induced spacings $\xi_{ni}$  (intervals of constancy of $\hat f_n$), consisting of two consecutive original spacings between locations of jumps of the least concave majorant, etc., where $N_j$ can be zero.

We now have the following representation theorem:

\begin{theorem}[Theorem 2.1 in \cite{piet_ron:83}]
\begin{align*}
\left(nD_{n1},\dots,nD_{n,m+1};Q_{n1},\dots,Q_{n,m}\right)\stackrel{{\cal D}}=\left({\bm S}^{(n)},{\bm N}^{(n)}\bigm|S_n=n,T_n=n\right)
\end{align*}
\end{theorem}

\begin{remark}
{\rm
For specificity, the random variables $S_{ji},\,i=1,\dots,N_j$, and $D_{ji},\,i=1,\dots Q_{nj}$, are ordered in Theorem 2.1 in \cite{piet_ron:83}. There is also a zero-step spacing introduced in \cite{piet_ron:83}, but this does not seem to be necessary.
}
\end{remark}

Using this representation, we can reduce the proofs of the limit behavior of global functionals of the Grenander estimator to a theorem for gamma and Poisson random variables, under the condition $(T_n,S_n)=(n,n)$. For convenience in later proofs, we also use a further standardization of $(S_n,T_n)$:
\begin{align}
\label{def_V_n_W_n}
V_n=n^{-1/2}\biggl\{S_n-\sum_{j=1}^n jN_j\biggr\}=n^{-1/2}\left\{S_n-T_n\right\},\qquad W_n=\frac{T_n}{n},
\end{align}
where $S_n$ and $T_n$ are defined by (\ref{def_S_n_T_n}). A conditional central limit theorem for functionals of the Grenander estimator can then be proved under the condition:
\begin{align}
\label{V_n-W_n-conditions}
V_n=0,\qquad W_n=1.
\end{align}

The infinitely divisible limit distribution of the pair $(V_n,W_n)$ was given in Lemma 3.1 of \cite{piet_ron:83}, but unfortunately Lemma 3.1 of \cite{piet_ron:83} contains a rather silly error (the $u^2$ in the representation of the characteristic function should be $u$). The correct version of this result is given in Lemma \ref{lemma_V_W} below, where also the origin of the error is explained.
The proof in \cite{GroLo:93} does not use the result on the limit distribution of $(V_n,W_n)$, so is not influenced by the erroneous Lemma 3.1 in \cite{piet_ron:83}.

We use methods different from those in \cite{piet_ron:83}. The conditional central limit theorem was proved in \cite{piet_ron:83} using Le Cam's paper \cite{lecam:58} (a paper apparently published without his permission, as became clear in a conversation of the author with him). In the present context, where we clearly have to deal with non-standard asymptotics, \cite{lecam:58} does not seem to be the right tool to use. We replace this by a direct analysis of the characteristic function.
The crucial tools here are Cauchy's formula and saddle point methods, using contour integration in the complex plane. To illustrate our method in a simple setting, we give a shortened version of the proof in \cite{GroLo:93} of Sparre Andersen's result \cite{sparre:54} in Section \ref{sec:Sparre_Andersen}.

\section{Sparre Andersen's result}
\label{sec:Sparre_Andersen}
To illustrate our method in the simplest setting, we give a short version of the proof in \cite{GroLo:93} of the following remarkable result of Sparre Andersen in \cite{sparre:54}.

\begin{theorem}[Sparre Andersen's result]
\label{th:Sparre_Andersen}
\end{theorem}
Let $X_1,\dots,X_n$ be a sample from the Uniform$(0,1)$ distribution and let $N_{jumps}$ be the number of jumps of the Grenander estimator for this sample. Then
\begin{align*}
\frac{N_{jumps}-\log n}{\sqrt{\log n}}\stackrel{{\cal D}}\longrightarrow N(0,1),
\end{align*}
where $N(0,1)$ is the standard normal distribution.

\begin{proof}
Let, for the sample $X_1,\dots,X_n$, $U_n$ be defined by:
\begin{align*}
U_n=\frac{\sum_{j=1}^n N_j-\log n}{\sqrt{\log n}}\,,
\end{align*}
and let $T_n$ be defined as in (\ref{def_S_n_T_n}).
Using (part of) the representation, introduced in Section \ref{sec:intro}, we only have to prove that $(U_n|T_n=n)$ tends in law to a standard normal distribution.
To this end we consider the conditional characteristic function
\begin{align*}
\E\left\{e^{isU_n}\bigm|T_n=n\right\}.
\end{align*}
Lemma 3.2 of \cite{piet_ron:83} implies:
\begin{align*}
\P\left\{T_n=n\right\}=\exp\biggl\{-\sum_{j=1}^n\frac1j\biggr\}.
\end{align*}
Hence we get, by Fourier inversion and using the notation $b_n=\sqrt{\log n}$,
\begin{align*}
&\E\left\{e^{isU_n}\bigm|T_n=n\right\}=\exp\Biggl\{\sum_{j=1}^n\frac1j\Biggr\}\frac1{2\pi}\int_{u=-\pi}^{\pi}\E\,e^{isU_n+iu T_n-inu}\,du\\
&=\exp\Biggl\{\sum_{j=1}^n\frac1j\Biggr\}\frac1{2\pi}\int_{u=-\pi}^{\pi}\exp\Biggl\{-isb_n+\sum_{j=1}^n\frac1j\left(e^{iju+is/b_n}-1\right)-inu\Biggr\}\,du\\
&=e^{-isb_n}\frac1{2\pi}\int_{u=-\pi}^{\pi}\exp\Biggl\{e^{is/b_n}\sum_{j=1}^n\frac1j e^{iju}-inu\Biggr\}\,du.
\end{align*}
Denoting the contour $u\mapsto e^{iu},\,u\in[-\pi,\pi),$ by $C$, we write this in the form
\begin{align}
\label{Cauchy1}
e^{-isb_n}\frac1{2\pi i}\int_{C}\exp\left\{\d_n\sum_{j=1}^n\frac{z^j}{j}\right\}z^{-n-1}\,dz,
\end{align}
where $\d_n$ is given by:
\begin{align*}
\d_n=e^{is/b_n}\,.
\end{align*}
The expression in (\ref{Cauchy1}), multiplied by $e^{isb_n}$, is by Cauchy's formula equal to the coefficient of $z^n$ in the power series around $z=0$ of the function:
\begin{align*}
z\mapsto \exp\left\{\d_n\sum_{j=1}^n\frac{z^j}{j}\right\}.
\end{align*}
Comparing this with the power series of the function $z\mapsto (1-z)^{-\d_n}$, we see that the coefficient of $z^n$ is the same in both series. This coefficient is:
\begin{align*}
(-1)^n{-\d_n\choose n}=\prod\left(1+\frac{\d_n-1}{j}\right)=\exp\Biggl\{\sum_{j=1}^n\log\left(1+\frac{\d_n-1}{j}\right)\Biggr\}.
\end{align*}
So we get:
\begin{align*}
&\E\left\{e^{isU_n}\bigm|T_n=n\right\}=\exp\Biggl\{-isb_n+\sum_{j=1}^n\log\left(1+\frac{\d_n-1}{j}\right)\Biggr\}\\
&=\exp\Biggl\{-isb_n+\sum_{j=1}^n\log\left(1+\frac{e^{is/b_n}-1}{j}\right)\Biggr\}\\
&=\exp\Biggl\{-isb_n+\sum_{j=1}^n\frac{e^{is/b_n}-1}{j}+o(1)\Biggr\}=\exp\Biggl\{-\frac{s^2}{2b_n^2}\sum_{j=1}^n\frac1j+o(1)\Biggr\}\\
&=\exp\left\{-\tfrac12s^2+o(1)\right\}.
\end{align*}
\end{proof}

\section{The limit distribution of the conditioning variables $(V_n,W_n)$}
\label{sec:limit_V_W}
Let the pair $(V_n,W_n)$ be defined by (\ref{def_V_n_W_n}). We prove the following lemma, which corrects Lemma 3.1 in \cite{piet_ron:83}.

\begin{lemma}
\label{lemma_V_W}
The pair $(V_n,W_n)$ converges in distribution to $(V,W)$, where $(V,W)$ has the infinitely divisible characteristic function
\begin{align*}
\f_{(V,W)}(t,u)=\exp\left\{\int_0^1\frac{e^{-\left(\tfrac12t^2-iu\right)y}-1}{y}\,dy\right\}.
\end{align*}
\end{lemma}

\begin{proof}
We have:
\begin{align}
\label{Poissonized_charfu}
\E\exp\left\{itV_n+iuW_n\right\}=&\exp\left\{\sum_{j=1}^n\frac{e^{i uj/n}\f_{S_{j1}-j}(tn^{-1/2})-1}{j}\right\}\nonumber\\
&=\exp\left\{\sum_{j=1}^n\frac{e^{i j (u/n-tn^{-1/2})}\left(1-itn^{-1/2}\right)^{-j}-1}{j}\right\},
\end{align}
where $\f_{S_{j1}-j}$ is the characteristic function of the centered gamma variable $S_{j1}-j$, see (3.9) of \cite{piet_ron:83}.

Writing $y_{j,n}=j/n$, and noting that for $y\in(0,1)$:
\begin{align*}
\frac{e^{i ny(u/n-tn^{-1/2})}\left(1-itn^{-1/2}\right)^{-ny}-1}{y}=\frac{e^{iyu-\tfrac12t^2y}-1}{y}-\tfrac13i e^{-\tfrac12t^2 y + i u y} t^3n^{-1/2}+O\left(n^{-1}\right),
\end{align*}
and that the limit of the expression on the left for $y\downarrow0$ is equal to:
\begin{align*}
i(-t\sqrt{n} + u) - n \log(1 - it/\sqrt{n})=iu-\tfrac12t^2+O\left(n^{-1/2}\right),\qquad n\to\infty,
\end{align*}
we can write the exponent in the form
\begin{align*}
\sum_{j=1}^n\frac{e^{-\tfrac12t^2y_{j,n}+iuy_{j,n}}-1}{ny_{j,n}}+O\left(n^{-1/2}\right),
\end{align*}
(it is here that the mistake was made in \cite{piet_ron:83}, in the formula after (3.9) on p.\ 333), so we get a Riemann sum converging to the integral
\begin{align*}
\int_0^1\frac{e^{-\tfrac12t^2y+iuy}-1}{y}\,dy.
\end{align*}
The infinite divisibility of the limit distribution is shown below.
\end{proof}

\begin{remark}
\label{remark:error_piet_ron}
{\rm In \cite{piet_ron:83} first the limit distribution of $U_n$ is computed, using moment conditions (going up to the $8$th moments). Next the limit distribution of $(V_n,W_n)$ is computed and it is stated that this distribution is infinite divisible and has no normal component, implying that therefore the limit  $(V_n,W_n)$ has to be independent of the limit of $U_n$.

The $s^2u^2$ in the exponent of the characteristic function of the limit distribution of $(V_n,W_n)$ in Lemma 3.1 of \cite{piet_ron:83} should be $s^2u$. The incorrect $u^2$ arose on p.\ 333 of \cite{piet_ron:83}, where the limit of the characteristic function of the rescaled gamma random variable $(S_j-j)/\sqrt{n}$ was given by $\exp\{-s^2u^2/2\}$ instead of $\exp\{-s^2u/2\}$. This also invalidates the ensuing remarks on p.\ 333 of \cite{piet_ron:83}. We correct these remarks below.
}
\end{remark}

The distribution of $(V,W)$ is infinitely divisible, as we now show. A general characterization of infinitely divisible distributions in $\R^d$ is given in \cite{Sato:01} and given below for convenience.

\begin{theorem}[Theorem 1.3 in \cite{Sato:01}, L\'evy-Khintchine representation]
If the distribution $\m$ is infinitely divisible, then its characteristic function $\hat\m(\bms)=\int_{\R^d}\exp\{i\langle \bms,\bmz\rangle\}\,d\m(\bmz)$ is given by:
\begin{align}
\label{levy-khintchine_rep}
\hat\m(z)=\exp\left\{-\tfrac12\bms^T \bm A\bms+\int_{R^d}\left(e^{i\langle \bms,\bmz\rangle}-1-i\langle \bms,\bmz\rangle1_{\{\|\bmz\|\le1\}}(\bmz)\right)\,d\n(\bmz)+i\langle\bm\delta,\bms\rangle\right\},
\end{align}
where $\bm A$ is a symmetric nonnegative-definite $d\times d$ matrix, $\|\cdot\|$ is the Euclidean norm, $\n$ is a measure on $\R^d$ satisfying $\n(\{0\})=0$, $\int_{\R^d}\left(\|\bmz\|^2\wedge1\right)\,d\n(\bmz)<\infty$, and where $\bm\delta\in\R^d$. The representation (\ref{levy-khintchine_rep}) by $\bm A,\n$ and $\bm\delta$ is unique. Conversely, for any choice of $\bm A,\n$ and $\bm\delta$ satisfying the conditions above, there exists an infinite divisible distribution $\m$ having characteristic  function (\ref{levy-khintchine_rep}).
\end{theorem}

In the present situation we can take $\bmA$ the $2\times2$ matrix with zeroes, $\bm\delta=(0,0)^T$ and define $\n$ by the density
\begin{align*}
\frac{\partial^2\n(v,w)}{\partial v\partial w}=\f\left(v/\sqrt{w}\right)w^{-3/2}1_{(0,1)}(w),
\end{align*}
where $\f$ is the standard normal density. With these choices of $\bm A$, $\bm\delta$ and $\n$ we get, using the notation  $\bms=(t,u)^T$ and $\bmz=(v,w)^T$,
\begin{align*}
&\exp\left\{-\tfrac12\bms^T A\bms+\int_{R^2}\left(e^{i\langle \bms,\bmz\rangle}-1-i\langle \bms,\bmz\rangle1_{\{\|\bmz\|\le1\}}(\bmz)\right)\,d\n(\bmz)+i\langle\bm\delta,\bms\rangle\right\}\\
&=\exp\Biggl\{\int_{y=0}^1\frac{e^{-\tfrac12t^2y+iuy}-1}{y}\,dy\Biggr\}.
\end{align*}

We note that the computer package Mathematica evaluates the characteristic function for $(V,W)$ in the following way:
\begin{align}
\label{mathematica1}
\exp\Biggl\{\int_{y=0}^1\frac{e^{-\tfrac12t^2y+iuy}-1}{y}\,dy\Biggr\}
=\exp\left\{-\g-\Gamma(0,\tfrac12t^2-iu)-\log\left(\tfrac12t^2-iu\right)\right\},
\end{align}
where $\g$ is Euler's gamma and $\Gamma(0,\tfrac12t^2-iu)$ is the complementary incomplete gamma function, defined by:
\begin{align*}
\Gamma\left(0,\tfrac12t^2-iu\right)=\exp\left\{-\tfrac12t^2+iu\right\}\int_0^{\infty}\exp\left\{-\left(\tfrac12t^2-iu\right)x\right\}(1+x)^{-1}\,dx,
\qquad t\ne0.
\end{align*}
see, e.g., (2.01) on p.\ 109 of \cite{olver:74}.

\section{Central limit theorem for $\int h(\hat f_n(x))\,dx$}
\label{sec:bodhi}
In this section, we use the notation
\begin{align}
\label{b_n-c_n}
b_n=\sqrt{\tfrac34h''(1)^2\log n},\qquad c_n=\sqrt{n}.
\end{align}
Using the conditioning of Section \ref{sec:intro}, the statistic $\int h(\hat f_n(x))\,dx$ has the following representation:
\begin{align*}
\int h(\hat f_n(x))\,dx=n^{-1}\sum_{j=1}^n \sum_{i=1}^{N_j} h\left(\frac{j}{S_{ji}}\right)S_{ji},
\end{align*}
where  the Poisson random variables $N_j$ and the gamma random variables $S_{ji}$ are defined as in (\ref{def_S_n_T_n}), and where we condition on
$(V_n,W_n)=(0,1)$, where $(V_n,W_n)$ is defined by (\ref{def_V_n_W_n}). We define
\begin{align}
\label{def_U_n}
U_n=
\frac1{\sqrt{\tfrac34h''(1)^2\log n}}\Biggl\{\sum_{j=1}^n \sum_{i=1}^{N_j} \left[\left(h\left(\frac{j}{S_{ji}}\right)-h(1)\right)S_{ji}+h'(1)\left(S_{ji}-j\right)\right]-\tfrac12h''(1)\log n\Biggr\}.
\end{align}

\begin{remark}
\label{remark_effect_conditioning}
{\rm
The terms $h'(1)\left(S_{ji}-j\right)$ are present in $U_n$ as variance reducing terms and give, after the summation over $i$ and $j$, a zero contribution to $U_n$ if $(V_n,W_n)=(0,1)$. Also note that
\begin{align*}
h(1)\sum_{j=1}^n \sum_{i=1}^{N_j}S_{ji}=nh(1)=n\m(h,f),
\end{align*}
if the condition $(V_n,W_n)=(0,1)$ is satisfied.
}
\end{remark}

We assume that the function $h$ satisfies condition \ref{condition_h} and first consider the conditional density of $V_n$, given $W_n=1$.

\begin{lemma}
\label{lemma:conditional_dens}
The conditional density of $V_n$, given $W_n=1$, is the density of a centered and standardized Gamma$(n,1)$ variable:
\begin{align*}
f_{V_n|W_n=1}(x)=\Gamma(n)^{-1}n^{n-1/2}e^{-n-x\sqrt{n}}\left(1+x/\sqrt{n}\right)^{n-1},\,x\in\R.
\end{align*}
The density $f_{V_n|W_n=1}(x)$ converges uniformly to the standard normal density, as $n\to\infty$.

\end{lemma}

\begin{proof}
By Fourier inversion, the conditional characteristic function is given by:
\begin{align*}
&\frac1{\P\{W_n=1\}}\frac1{2\pi}\int_{-\pi}^{\pi}\E\exp\left\{it V_n+niuW_n-niu\right\}\,du\\
&=\frac1{\P\{W_n=1\}}\frac1{2\pi}\int_{-\pi}^{\pi}\exp\left\{\sum_{j=1}^n\frac{e^{i j (u-t/c_n)}\left(1-it/c_n\right)^{-j}-1}{j}-niu\right\}\,du,
\end{align*}
where $c_n=\sqrt{n}$, see (\ref{Poissonized_charfu}). Denoting the contour $w\mapsto e^{iw},\,w\in[-\pi,\pi),$ by $C$, we write this in the form
\begin{align}
\label{Cauchy_asymptotics}
&P\{W_n=1\}^{-1}\exp\Biggl\{-\sum_{j=1}^n\frac1j\Biggr\}\frac1{2\pi i}\int_{C}\exp\left\{\sum_{j=1}^n\frac{(\b_n(t) z)^j}{j}\right\}z^{-n-1}\,dz\nonumber\\
&=\frac1{2\pi i}\int_{C}\exp\left\{\sum_{j=1}^n\frac{(\b_n(t) z)^j}{j}\right\}z^{-n-1}\,dz,
\end{align}
where $\b_{n}(t)$ is given by:
\begin{align*}
\b_{n}(t)=e^{-\frac{it}{c_n}}\left(1-\frac{it}{c_n}\right)^{-1}\,,
\end{align*}
and where we use:
\begin{align*}
\P\{W_n=1\}=\exp\Biggl\{-\sum_{j=1}^n\frac1j\Biggr\},
\end{align*}
by Lemma 3.2 of \cite{piet_ron:83}.
An application of Cauchy's formula yields:
\begin{align*}
\frac1{2\pi i}\int_{C}\exp\left\{\sum_{j=1}^n\frac{(\b_n(t) z)^j}{j}\right\}z^{-n-1}\,dz=\b_n(t)^n,
\end{align*}
where we use that the coefficient of $z^n$ in the power series around $z=0$ of the function $z\mapsto\exp\{\sum_{j=1}^n (\b_n(t)z)^j/j\}$ is the same as the coefficient of $z^n$ in the power series of the function $z\mapsto (1-\b_n(t) z)^{-n}$.

Hence
\begin{align*}
&\E\left\{e^{it V_n}\Bigm|W_n=1\right\}=\frac1{\P\{W_n=1\}}\frac1{2\pi}\int_{-\pi}^{\pi}\E\exp\left\{it V_n+niuW_n-niu\right\}\,du\\
&=\b_n(t)^n=\exp\left\{-nit/c_n\right\}\left(1-\frac{it}{c_n}\right)^{-n}.
\end{align*}
This is just the characteristic function of the sum of $n$ standardized exponential variables, and hence its density tends uniformly to the standard normal density by Theorem 2 on p.\ 516 of \cite{feller2:71}.
\end{proof}

\vspace{0.3cm}
So, in particular, we get:
\begin{align*}
\P\{W_n=1\}f_{V_n|W_n=1}(0)=\frac1{\sqrt{2\pi}}\exp\biggl\{-\sum_{j=1}^n 1/j\biggr\}(1+o(1))=\frac{e^{-\g}}{n\sqrt{2\pi}}(1+o(1)),\qquad n\to\infty,
\end{align*}
where $\g$ is Euler's gamma. The characteristic function of $(U_n|V_n=0,W_n=1)$ is therefore given by
\begin{align}
\label{main_conditional_charfu}
&\frac1{4\pi^2}\int_{t=-\infty}^{\infty}\int_{-\pi}^{\pi}\E e^{isU_n+itV_v+iuW_n-inu}\,du\,dt\bigm/\left(\P\{W_n=1\}f_{V_n|W_n=1}(0)\right)\nonumber\\
&\sim\frac{ne^{\g}}{(2\pi)^{3/2}}\int_{t=-\infty}^{\infty}\int_{-\pi}^{\pi}\E e^{isU_n+itV_v+iuW_n-inu}\,du\,dt.
\end{align}

We now consider the characteristic function $\f_{nj}$, defined by:
\begin{align}
\label{def_phi_nj}
\f_{nj}(s,t)=E\exp\left\{is\frac{\left\{h(j/S_{j1})-h(1)\right\}S_{j1}+h'(1)\left(S_{j1}-j\right)}{b_n}+it\frac{S_{j1}-j}{c_n}\right\}.
\end{align}
which involves the components of the random variables $U_n$ and $V_n$. Its asymptotic behavior is determined using a saddle point method.

\begin{lemma}
\label{lemma:charfu_phi_n}
The characteristic function (\ref{def_phi_nj}) satisfies
\begin{align}
\label{deBruijn_formula}
\f_{nj}(s,t)
\sim\frac1j\exp\left\{-\frac{jt^2}{2n}+O\left(\frac{j}{nb_n}\right)\right\}\left\{1+\frac{ish''(1)}{2b_n}-\frac{3s^2h''(1)^2}{8b_n^2}\right\},
\qquad j\to\infty.
\end{align}
uniformly for $t$ in a bounded interval.
\end{lemma}

\begin{proof}
After a change of variables, $\f_{nj}(s,t)$ can be written:
\begin{align*}
&\frac{j^j}{\Gamma(j)}\int_{x=0}^{\infty}\exp\left[j\left\{is\frac{\{h(1/x)-h(1)\}x+h'(1)(x-1)}{b_n}+\frac{it(x-1)}{c_n} -x+\log x\right\}\right]x^{-1}\,dx\nonumber\\
&=\frac{j^j}{\Gamma(j)}\int_{x=0}^{\infty}\exp\left\{jf_{n,s,t}(x)\right\}x^{-1}\,dx,
\end{align*}
where $f_{n,s,t}$ is defined on the right half plane by:
\begin{align}
\label{def_f_general}
f_{n,s,t}(z)=is\frac{\left\{h(1/z)-h(1)\}z+h'(1)(z-1)\right\}}{b_n}+\frac{it(z-1)}{c_n} -z+\log z,\qquad \text{\rm Re}(z)>0,
\end{align}
The derivative of $f_{n,s,t}$ is given by
\begin{align}
\label{derivative_f_general}
f_{n,s,t}'(z)=\frac{is\{h(1/z)-h(1)\}}{b_n}-\frac{is h'(1/z)}{b_nz}+\frac{ish'(1)}{b_n}+\frac{it}{c_n}-1+\frac1z.
\end{align}
A saddle point is given by the equation
\begin{align*}
f_{n,s,t}'(z)=0.
\end{align*}
Multiplying both sides of this equation with $z$, we get the equation
\begin{align}
\label{saddle_point_eq2}
z=g_n(z),\qquad g_n(z)\stackrel{\text{\small def}}=\left(1-\frac{it}{c_n}\right)^{-1}\left\{1+\frac{is\{h(1/z)-h(1)\}z}{b_n}-\frac{is h'(1/z)}{b_n}+\frac{ish'(1)z}{b_n}\right\}
\end{align}
This equation has, for sufficiently large $n$, a unique solution in a neighborhood of $z_0(t)=\left(1-\frac{it}{c_n}\right)^{-1}$, as is clear from the following properties.
\begin{enumerate}
\item[(i)]
\begin{align*}
g_n'(z)=\left(1-\frac{it}{c_n}\right)^{-1}b_n^{-1}\left\{ish(1)+ish(1/z)+ish'(1)-\frac{ish'(1/z)}{z}+\frac{ish''(1/z)}{z^2}\right\}\longrightarrow0,\qquad n\to\infty.
\end{align*}
\item[(ii)] For $z_0(t)=\left(1-\frac{it}{c_n}\right)^{-1}$ we have:
\begin{align*}
g_n(z_0(t))-z_0(t)=-\frac{is\left\{h(1)-h(1-it/c_n)-h'(1)/c_n+(1-it/c_n)h'(1-it/c_n)\right\}}{b_n(1-it/c_n)^2}
\longrightarrow0,\qquad n\to\infty,
\end{align*}
\end{enumerate}
see, e.g., (3.3), p.\ 56 of \cite{dieudonne:68}.
So the saddle point is given by a solution of the equation (\ref{saddle_point_eq2}) and can be found by the simple iteration $z_{k+1}=g_n(z_k)$, starting at $z_0(t)$. We can, however, also take $z_0(t)$ itself instead of the real saddle point for the asymptotic expansion, since this gives us the same terms in the expansion we need.

The value of $f_{n,s,t}$ at $z_0(t)$ has the following expansion:
\begin{align}
\label{expansion_f_n}
f_{n,s,t}(z_0(t))=-1-\frac{t^2+ist^2h''(1)/(2b_n)+o\left(b_n^{-1}\right)}{2n}=-1-\frac{t^2}{2n}+O\left(\frac1{nb_n}\right).
\end{align}

Furthermore,
\begin{align*}
f_{n,s,t}''(z)=\frac{-b_nz+ish''(1/z)}{b_nz^3}\,.
\end{align*}
so
\begin{align*}
f_{n,s,t}''(z_0(t))=-1+\frac{ish''(1)}{b_n}+O\left(c_n^{-1}\right).
\end{align*}
It follows that we get:
\begin{align*}
&\frac{\a_n}{|f_{n,s,t}''(z_0(t))|^{1/2}}\approx1+\frac{ish''(1)}{2b_n}-\frac{3h''(1)^2s^2}{8b_n^2}+O\left(c_n^{-1}\right)
=1+\frac{ish''(1)}{2b_n}-\frac{s^2}{2\log n}+O\left(c_n^{-1}\right),
\end{align*} 
where $\a_n=\exp\left\{ish''(1)/(2b_n)\right\}$ is a complex number with absolute value 1, corresponding to the argument of the main axis of the saddle point (note that the argument of this axis is $\tfrac12\pi-\tfrac12\arg f_{n,s,t}''(z_0(t))$, see \cite{debruijn:81}, p.\ 84).

Evaluating the integrand at $z_0(t)$, and applying Stirling's formula on $\Gamma(j)$, we obtain the following asymptotic representation:
\begin{align}
\label{deBruijn_formula1a}
\f_{nj}(s,t)&=\frac{j^j}{\Gamma(j)}\int_{x=0}^{\infty}\exp\left\{jf_{n,s,t}(x)\right\}x^{-1}\,dx
\sim\frac{\a_n\sqrt{2\pi}e^{j\{f_{n,s,t}(z_0(t))+1\}}}{\sqrt{2\pi j}\sqrt{j |f_{n,s,t}''(z_0(t))|}}\nonumber\\
&\sim\frac1j\exp\left\{-\frac{jt^2}{2n}+O\left(\frac{j}{nb_n}\right)\right\}\left\{1+\frac{ish''(1)}{2b_n}-\frac{3s^2h''(1)^2}{8b_n^2}\right\},\qquad j\to\infty.
\end{align}
see, e.g., \cite{debruijn:81}, (5.10.3) on p.\ 92 for the first asymptotic equivalence. The second asymptotic equivalence holds uniformly for $t$ in a bounded interval. Note that this corresponds to changing the path of integration for $x$ to a path in the complex plane, going through the saddle point.
\end{proof}

\vspace{0.3cm}
We can now prove the following property of the characteristic function of $(U_n,V_n,W_n)$. This is ``almost'' the Fourier inversion for $(V_n,W_n)$, but we still have to extend the inversion for $V_n$ to the whole real line. Cauchy's formula is an essential ingredient of the proof of Lemma \ref{lemma:region_A}.

\begin{lemma}
\label{lemma:region_A}
Let $U_n$ be defined by (\ref{def_U_n}) and let $h$ satisfy condition \ref{condition_h}.
Then, for each $M>0$:
\begin{align*}
\frac{ne^{\g}}{(2\pi)^{3/2}}\int_{t=-M}^{M}\int_{-\pi}^{\pi}\E e^{isU_n+itV_v+niuW_n-inu}\,du\,dt\longrightarrow 
\frac{e^{-\tfrac12s^2}}{\sqrt{2\pi}}\int_{t=-M}^M e^{-\tfrac12t^2}\,dt,\qquad n\to\infty,
\end{align*}
where $\g$ is Euler's gamma.
\end{lemma}

\begin{proof}
Let $\f_{nj}$ be defined by (\ref{def_phi_nj}). Since, by (\ref{b_n-c_n}) and (\ref{def_U_n}), 
\begin{align*}
U_n&=
\frac1{\sqrt{\tfrac34h''(1)^2\log n}}\Biggl\{\sum_{j=1}^n \sum_{i=1}^{N_j} \left[\left(h\left(\frac{j}{S_{ji}}\right)-h(1)\right)S_{ji}+h'(1)\left(S_{ji}-j\right)\right]-\tfrac12h''(1)\log n\Biggr\}\\
&=b_n^{-1}\Biggl\{\sum_{j=1}^n \sum_{i=1}^{N_j} \left[\left(h\left(\frac{j}{S_{ji}}\right)-h(1)\right)S_{ji}+h'(1)\left(S_{ji}-j\right)\right]-\tfrac12h''(1)\log n\Biggr\},
\end{align*}
where
\begin{align*}
b_n=\sqrt{\tfrac34h''(1)^2\log n},
\end{align*}
we get, evaluating the probabilities for the Poisson random variables $N_j$ in the third line:
\begin{align}
\label{asymptotic_charfu}
&\E\exp\left\{i s U_n+i t V_n+i u W_n\right\}=\exp\left\{-\tfrac12ish''(1)(\log n)/b_n\right\}\E\left\{\prod_{j=1}^n\left\{\f_{nj}(s,t)\right\}^{N_j} e^{iuN_j/n}\right\}\nonumber\\
&=\exp\Biggl\{-\tfrac12ish''(1)(\log n)/b_n+\sum_{j=1}^n\frac1j\left\{\f_{nj}(s,t)e^{iju/n}-1\right\}\Biggr\}\nonumber\\
&=\exp\left[-\tfrac12ish''(1)(\log n)/b_n+\sum_{i=1}^n\frac1j\left\{\exp\left(-\frac{jt^2}{2n}+\frac{iju}{n}\right)-1\right\}\right.\nonumber\\
&\qquad\qquad\qquad\qquad\qquad\qquad\qquad\qquad+\left.\left\{\frac{ish''(1)}{2b_n}-\frac{3s^2h''(1)^2}{8b_n^2}\right\}\sum_{j=1}^n\frac1j\exp\left(-\frac{jt^2}{2n}+\frac{ij u}{n}\right)+o(1)\right].
\end{align}

As in the proof of Lemma \ref{lemma:conditional_dens} we consider the contour $w\mapsto e^{iw},\,w\in[-\pi,\pi),$ and denote this contour by $C$. So, integrating (\ref{asymptotic_charfu}) w.r.t.\ $u$ and changing variables we get:
\begin{align}
\label{Cauchy_asymptotics2}
&\frac1{2\pi}\int_{u=-\pi}^{\pi}\E\exp\left\{i s U_n+i t V_n+i n u W_n-inu\right\}\,du\\
&=\exp\Biggl\{-\tfrac12ish''(1)(\log n)/b_n-\sum_{j=1}^n\frac1j\Biggr\}\frac1{2\pi i}\int_{C}\exp\left\{\sum_{j=1}^n\frac{\d_n(\b_n(t) z)^j}{j}\right\}z^{-n-1}\,dz,
\end{align}
where
\begin{align}
\label{def_beta_n}
\b_n(t)\sim\exp\left(-\frac{t^2}{2n}\right),
\end{align}
and
\begin{align}
\label{def_delta_n}
\d_n(t)=\frac{\a_n}{\sqrt{|f_{n,s,t}''(z_0(t))|}}=1+\frac{ish''(1)}{2b_n}-\frac{3s^2h''(1)^2}{8b_n^2}+O\left(b_n^{-3}\right),
\end{align}
and where $\a_n$ is a complex number of absolute value 1, corresponding to the angle of the main axis through the (approximate) saddle point $z_0(t)$.

So we get, by Cauchy's formula, for $t\in[-M,M]$ and $M$ arbitrarily large,
\begin{align}
\label{fundamental_equivalence}
&\frac1{2\pi}\int_{-\pi}^{\pi}\E\exp\left\{i s U_n+i t V_n+i nu W_n-niu\right\}\,du\nonumber\\
&\sim \exp\Biggl\{-\tfrac12ish''(1)(\log n)/b_n-\sum_{j=1}^n\frac1j\Biggr\}(-1)^n{-\d_n(t)\choose n}\b_n(t)^n\nonumber\\
&=\exp\Biggl\{-\tfrac12ish''(1)(\log n)/b_n-\sum_{j=1}^n\frac1j\Biggr\}\prod_{j=1}^n\left(1+\frac{\d_n(t)-1}{j}\right)\b_n(t)^n\nonumber\\
&\sim\exp\Biggl\{-\tfrac12ish''(1)(\log n)/b_n-\sum_{j=1}^n\frac1j\Biggr\}\exp\left\{-\tfrac12t^2+\sum_{j=1}^n\log\left(1+\frac{\d_n(t)-1}j\right)\right\}\nonumber\\
&=\exp\Biggl\{-\tfrac12ish''(1)(\log n)/b_n-\log n-\g\Biggr\}\exp\left\{-\tfrac12t^2+\{\d_n(t)-1\}\sum_{j=1}^n\frac1j+o(1)\right\}\nonumber\\
&=\frac1n\exp\Biggl\{-\tfrac12ish''(1)(\log n)/b_n-\g\Biggr\}\exp\left\{-\tfrac12t^2+\left\{\frac{ish''(1)}{2b_n}-\frac{3s^2h''(1)^2}{8b_n^2}\right\}\sum_{j=1}^n\frac1j+o(1)\right\}.
\end{align}

Thus:
\begin{align*}
&\frac{ne^{\g}}{(2\pi)^{3/2}}\int_{t=-M}^{M}\int_{-\pi}^{\pi}\E e^{isU_n+itV_v+niuW_n-inu}\,du\,dt\\
&=\frac1{\sqrt{2\pi}}\int_{t=-M}^M\exp\left\{-\tfrac12t^2-\frac{3s^2h''(1)^2}{8b_n^2}\sum_{j=1}^n\frac1j+o(1)\right\}\,dt\\
&=\exp\left\{-\tfrac12s^2+o(1)\right\}\frac1{\sqrt{2\pi}}\int_{t=-M}^M\exp\left\{-\tfrac12t^2\right\}\,dt.
\end{align*}
\end{proof}

We still have to prove that the remaining part of integral w.r.t.\ the integration variable $t$ can be made arbitrarily small by choosing $M$ large. To this end we split the remaining region into two regions: $A_1=\{t\in\R:M<|t|\le\d n^{1/2}\}$ and $A_2=\{t\in\R:|t|>\d n^{1/2}\}$. This split-up is familiar from inversion theorems for densities, see, e.g., the proof of Theorem 2 on p.\ 516 of \cite{feller2:71}. 
We start with the region  $A_1=\{t\in\R:M<|t|\le \d n^{1/2}\}$.

\begin{lemma}
Let $h$ satisfy condition \ref{condition_h}. Then there exists for each $\e>0$ an $M>0$ and $\d>0$ such that
\begin{align*}
\left|\int_{t:M<|t|\le \d n^{1/2}}\int_{u=-\pi}^{\pi}\E\exp\left\{i s U_n+i tV_n+i nu W_n-niu\right\}\,du\,dt\right|<\e,
\end{align*}
\end{lemma}

\begin{proof}
We consider again the expansion of the function $f_{n,s,t}$ defined by (\ref{def_f_general}) at the point $z_0(t)=(1-it/c_n)^{-1}$.
We get:
\begin{align*}
f_{n,t,u}\left(z_0(t)\right)=-1-\frac{t^2}{2n}-\frac{ist^2h''(\theta)}{2nb_n},
\end{align*}
where $\th$ is a point on the line segment between $1$ and $1/(1-it/c_n)$. Likewise
\begin{align*}
f_{n,t,u}''\left(z_0(t)\right)=-1+\frac{ish''(1)}{b_n}+\frac{2i\th'}{c_n},
\end{align*}
where $\th'$ is a point on the line segment between $1$ and $1/(1-it/c_n)$. 

So we have a local expansion
\begin{align}
\label{deBruijn_formula1b}
\f_{nj}(s,t)\sim\frac1j\exp\left\{-\frac{jt^2}{2n}+O\left(\frac{jt^2}{nb_n}\right)\right\}\left\{1+\frac{ish''(1)}{2b_n}-\frac{3s^2h''(1)^2}{8b_n^2}+O\left(c_n^{-1}\right)\right\},\qquad j\to\infty.
\end{align}
as in (\ref{deBruijn_formula1a}). This means that we can follow the same steps as in the proof of Lemma \ref{lemma:region_A} and that we can choose $M$ and $\d>0$ in such a way that
\begin{align*}
&\left|\frac{ne^{\g}}{(2\pi)^{3/2}}\int_{M<|t|\le \d n^{1/2}}\int_{-\pi}^{\pi}\E e^{isU_n+itV_v+niuW_n-inu}\,du\,dt\right|\\
&\le\exp\left\{-\tfrac12s^2+o(1)\right\}\frac1{\sqrt{2\pi}}\int_{M<|t|\le \d n^{1/2}}\exp\left\{-\tfrac12t^2+\tfrac14t^2\right\}\,dt\\
&=\exp\left\{-\tfrac12s^2+o(1)\right\}\frac1{\sqrt{2\pi}}\int_{M<|t|\le \d n^{1/2}}\exp\left\{-\tfrac14t^2\right\}\,dt<\e.
\end{align*}
\end{proof}

The following lemma deals with the region  $A_2=\{t\in\R:|t|>\d n^{1/2}\}$.
\begin{lemma}
\label{lemma:region_A_2}
Let $h$ satisfy condition \ref{condition_h}.
Then, for each $\d>0$:
\begin{align*}
\int_{|t|>\d c_n}\int_{u=-\pi}^{\pi}\E\exp\left\{i s U_n+i tV_n+i nu W_n-niu\right\}\,du\,dt\longrightarrow 0,\qquad n\to\infty.
\end{align*}
\end{lemma}
\begin{proof}
We consider the characteristic function:
\begin{align*}
\bar\f_{nj}(s,t)=E\exp\left\{is\frac{\left\{h(j/S_{j1})-h(1)\right\}S_{j1}+h'(1)\left(S_{j1}-j\right)}{b_n}+it\left(S_{j1}-j\right)\right\},
\end{align*}
so we replace $c_n$ by $1$ in (\ref{def_phi_nj}). This means that, for the saddle point analysis, the constant $c_n$ is replaced by $1$ in the function (\ref{def_f_general}). So we now define
\begin{align}
\label{def_f_general2}
\bar f_{n,s,t}(z)=is\frac{\left\{h(1/z)-h(1)\}z+h'(1)(z-1)\right\}}{b_n}+it(z-1) -z+\log z,\qquad \text{\rm Re}(z)>0,
\end{align}
The saddle point equation (\ref{saddle_point_eq2}) now turns into
\begin{align}
\label{saddle_point_eq3}
z=g_n(z),\qquad g_n(z)\stackrel{\text{\small def}}=\left(1-it\right)^{-1}\left\{1+\frac{is\{h(1/z)-h(1)\}z}{b_n}-\frac{is h'(1/z)}{b_n}+\frac{ish'(1)z}{b_n}\right\}
\end{align}
and has a unique solution in a neighborhood of $(1-it)^{-1}$ for the same reasons as before.

We define:
\begin{align*}
z_0(t)=1/(1-it).
\end{align*}
Then:
\begin{align*}
\bar f_{n,s,t}(z_0(t))&=-\frac{t^2+1}{1-it}-\log(1-it)-is\frac{h(1)-h(1-it)+sth'(1)}{b_n(1-it)}\\
&=-\frac{t^2+1}{1-it}+O\left(b_n^{-1}|t|\right),
\end{align*}
and
\begin{align*}
\bar f_{n,s,t}''(z_0(t))&=(t+i)^2-\frac{s(t+i)h''(1-it)}{b_n}\\
&=(t+i)^2+O\left(b_n^{-1}|t|\right),
\end{align*}
uniformly in $|t|>\d$, using Condition \ref{condition_h}.
This implies that, uniformly for $|t|>\d$,
\begin{align}
\label{deBruijn_formula2}
\f_{nj}(s,t)&=\frac{j^j}{\Gamma(j)}\int_{x=0}^{\infty}\exp\left\{jf_{n,s,t}(x)\right\}x^{-1}\,dx
\sim\frac{\a_n(t)e^{j\{\bar f_{n,s,t}(z_0(t))+1\}}}{j\sqrt{|\bar f_{n,s,t}''(z_0(t))|}}\nonumber\\
&=\frac{\a_n(t)\exp\left\{-j\left(it+\log(1-it)+O\left(b_n^{-1}\right)\right)\right\}}{j\left|(t+i)^2+O\left(b_n^{-1}|t|\right)\right|^{1/2}},\qquad j\to\infty,
\end{align}
where $\a_n(t)$ is a complex number with absolute value 1 and argument $\tfrac12\pi-\tfrac12\arg \bar f_{n,s,t}''(z_0(t))$.

So we find, using Cauchy's formula again, as in the proof of Lemma \ref{lemma:region_A}, and using Condition \ref{condition_h},
\begin{align*}
&\frac1{2\pi}\int_{-\pi}^{\pi}\E\exp\left\{i s U_n+i tc_n V_n+i nu W_n-niu\right\}\,du\nonumber\\
&\sim \exp\Biggl\{-\tfrac12ish''(1)(\log n)/b_n-\sum_{j=1}^n\frac1j\Biggr\}(-1)^n{- \d_n(t)\choose n}\b_n(t)^n\nonumber\\
&=\exp\Biggl\{-\tfrac12ish''(1)(\log n)/b_n-\sum_{j=1}^n\frac1j\Biggr\}\prod_{j=1}^n\left(1+\frac{\d_n(t)-1}{j}\right)\b_n(t)^n,
\end{align*}
where
\begin{align*}
\b_n(t)=\exp\left\{-(it+\log(1-it)+O\left(b_n^{-1}\right)\right\},
\end{align*}
and
\begin{align*}
\d_n(t)=\frac{\a_n}{\left|(t+i)^2+O\left(b_n^{-1}|t|\right)\right|^{1/2}}\,.
\end{align*}
Hence
\begin{align*}
&\exp\Biggl\{-\tfrac12ish''(1)(\log n)/b_n-\sum_{j=1}^n\frac1j\Biggr\}\prod_{j=1}^n\left(1+\frac{\d_n(t)-1}{j}\right)\b_n(t)^n\\
&=\frac1n\exp\Biggl\{-\tfrac12ish''(1)(\log n)/b_n-\g\Biggr\}\\
&\qquad\qquad\cdot\exp\left\{-n\left\{it+\log(1-it)+O\left(b_n^{-1}\right)\right\}+\frac{\a_n \sum_{j=1}^n1/j}{\left|(t+i)^2+O\left(b_n^{-1}|t|\right)\right|^{1/2}}\right\},
\end{align*}
implying
\begin{align*}
&\left|\frac1{2\pi}\int_{-\pi}^{\pi}\E\exp\left\{i s U_n+i tc_n V_n+i nu W_n-niu\right\}\,du\right|\\
&\le\frac1n\exp\left\{-n\left\{\tfrac12\log(1+t^2)+O\left(b_n^{-1}\right)\right\}+\frac{\sum_{j=1}^n1/j}{\left|(t+i)^2+O\left(b_n^{-1}|t|\right)\right|^{1/2}}\right\}\\
&\le \left|1+t^2\right|^{-n/4},
\end{align*}
for large $n$, uniformly for $|t|>\d$, using Condition \ref{condition_h}.
 It now follows that
\begin{align*}
&\left|\frac1{2\pi}\int_{|t|>\d c_n}\int_{u=-\pi}^{\pi}\E\exp\left\{i s U_n+i tV_n+i nu W_n-niu\right\}\,du\,dt\right|\\
&\left|\frac1{2\pi}\int_{|t|>\d}\int_{u=-\pi}^{\pi}\E\exp\left\{i s U_n+i tc_nV_n+i nu W_n-niu\right\}\,du\,dt\right|\\
&\le \int_{|t|>\d}(1+t^2)^{-n/4}\,dt\longrightarrow0,\qquad n\to\infty.
\end{align*}
\end{proof}

This leads to the main result of this section.

\begin{theorem}
\label{th:main_theorem}
Let $h$ satisfy Condition \ref{condition_h}. Then $U_n|(V_n,W_n)=(0,1)$ converges in law to a standard normal distribution.
\end{theorem}

\begin{proof}
The preceding lemma's imply
\begin{align*}
\frac{ne^{\g}}{(2\pi)^{3/2}}\int_{t=-\infty}^{\infty}\int_{-\pi}^{\pi}\E e^{isU_n+itV_v+niuW_n-inu}\,du\,dt\longrightarrow 
\frac{e^{-\tfrac12s^2}}{\sqrt{2\pi}}\int_{t=-\infty}^{\infty} e^{-\tfrac12t^2}\,dt=e^{-\tfrac12s^2},\qquad n\to\infty,
\end{align*}
where $\g$ is Euler's gamma. Hence
\begin{align*}
\E\left\{e^{isU_n}\bigm|(V_n,W_n)=(0,1)\right\}\sim\frac{ne^{\g}}{(2\pi)^{3/2}}\int_{t=-\infty}^{\infty}\int_{-\pi}^{\pi}\E e^{isU_n+itV_v+niuW_n-inu}\,du\,dt
\longrightarrow e^{-\tfrac12s^2},\qquad n\to\infty.
\end{align*}
\end{proof}

Theorem \ref{main_theorem} now follows from  the conditional representation from Section \ref{sec:intro}, definition (\ref{def_U_n}) of $U_n$, Remark \ref{remark_effect_conditioning} and Theorem \ref{th:main_theorem}.

\section{Conclusion}
\label{sec:conclusion}
We derived a general theorem (Theorem \ref{main_theorem}) for integrals of the Grenander estimator when the distribution is uniform from a representation in terms of Poisson and gamma random variables in \cite{piet_ron:83}. The result implies the main result of \cite{piet_ron:83} and gives also the limit behavior of the entropy functional. We corrected the limit distribution of the conditioning variables given in Lemma 3.1 of \cite{piet_ron:83} in Section \ref{sec:limit_V_W}. The methods used are rather different from the methods in \cite{piet_ron:83}, where a result in \cite{lecam:58} was used.

Our main result was inspired by \cite{bodhi:19} who derived a similar result from \cite{piet_ron:83} under different conditions. The main tools are Cauchy's formula and the saddle point method for integrals of analytic functions of a complex variable. A simple version of the approach is given in Section \ref{sec:Sparre_Andersen} to illustrate the method without the complications of the saddle point method.

\bibliographystyle{plainnat}
\bibliography{cupbook}
\end{document}